\documentclass[3p,review,number,sort&compress]{elsarticle}
\usepackage{hyperref}
\usepackage{amssymb}
\usepackage{amsmath}

\newtheorem{theorem}{Theorem}[section]
\newtheorem{lemma}[theorem]{Lemma}
\newtheorem{corollary}[theorem]{Corollary}
\newtheorem{remark}[theorem]{Remark}
\newproof{proof}{Proof}
\numberwithin{equation}{section}

\newcommand{\cB}{\mathcal{B}}
\newcommand{\cP}{\mathcal{P}}
\newcommand{\cW}{\mathcal{W}}
\newcommand{\C}{\mathbb{C}}
\newcommand{\D}{\mathbb{D}}
\newcommand{\I}{\mathbb{I}}
\newcommand{\N}{\mathbb{N}}
\newcommand{\T}{\mathbb{T}}

\newcommand{\Z}{\mathbb{Z}}
\newcommand{\eps}{\varepsilon}

\begin{document}
\begin{frontmatter}
\title{The Brown-Halmos Theorem for a Pair of Abstract Hardy Spaces}

\author[AK]{Alexei Karlovich\corref{Alexei}}
\ead{oyk@fct.unl.pt}

\author[ES]{Eugene Shargorodsky}
\ead{eugene.shargorodsky@kcl.ac.uk}

\cortext[Alexei]{Corresponding author}

\address[AK]{
Centro de Matem\'atica e Aplica\c{c}\~oes,
Departamento de Matem\'atica,
Faculdade de Ci\^encias e Tecnologia,\\
Universidade Nova de Lisboa,
Quinta da Torre,
2829--516 Caparica,
Portugal}

\address[ES]{%
Department of Mathematics,
King's College London,
Strand, London WC2R 2LS,
United Kingdom}
\begin{abstract}
Let $H[X]$ and $H[Y]$ be abstract Hardy spaces built upon Banach function 
spaces $X$ and $Y$ over the unit circle $\T$. We prove an analogue of the
Brown-Halmos theorem for Toeplitz operators $T_a$ acting from $H[X]$ to $H[Y]$
under the only assumption that the space $X$ is separable and the Riesz
projection $P$ is bounded on the space $Y$. We specify our results to the 
case of variable Lebesgue spaces $X=L^{p(\cdot)}$ and $Y=L^{q(\cdot)}$ and to 
the case of Lorentz spaces $X=Y=L^{p,q}(w)$, $1<p<\infty$, $1\le q<\infty$ 
with Muckenhoupt weights $w\in A_p(\T)$.
\end{abstract}

\end{frontmatter}
\section{Introduction}
For $1\le p\le\infty$, let $L^p:=L^p(\T)$ represent the standard Lebesgue
space on the unit circle $\T$ in the complex plane $\C$ with respect to the
normalized Lebesgue measure $dm(t)=|dt|/(2\pi)$.  For $f\in L^1$, let
\[
\widehat{f}(n):=\frac{1}{2\pi}
\int_{-\pi}^\pi f(e^{i\varphi})e^{-in\varphi}\,d\varphi,
\quad n\in\Z,
\]
be the sequence of the Fourier coefficients of $f$. For
$1\le p\le\infty$, the classical Hardy spaces $H^p$ are defined by
\[
H^p:=\big\{f\in L^p\ :\ \widehat{f}(n)=0\quad\mbox{for all}\quad n<0\big\}.
\]
Consider the operators $S$ and $P$, defined for a function $f\in L^1$ 
and an a.e. point $t\in\T$ by
\begin{equation}\label{eq:Cauchy-Riesz}
(Sf)(t):=\frac{1}{\pi i}\,\mbox{p.v.}\int_\T
\frac{f(\tau)}{\tau-t}\,d\tau,
\quad
(Pf)(t):=\frac{f(t)+(Sf)(t)}{2},
\end{equation}
respectively, where the integral is understood in the Cauchy principal 
value sense. The operator $S$ is called the Cauchy singular integral
operator. It is well known that the operators $P$ and $S$ are bounded
on $L^p$ if $p\in(1,\infty)$ and are not bounded on $L^p$ if 
$p\in\{1,\infty\}$ (see, e.g., \cite[Section~4.4]{BK97} or
\cite[Section~1.42]{BS06}). Note that using the elementary
equality
\[
\frac{e^{i\theta}}{e^{i\theta}-e^{i\vartheta}}
=
\frac{1}{2}\left(1+i\cot\frac{\vartheta-\theta}{2}\right),
\quad\theta,\vartheta\in[-\pi,\pi],
\]
one can write for $f\in L^1$ and $\vartheta\in[-\pi,\pi]$,
\[
(Sf)(e^{i\vartheta})
=
\frac{1}{\pi}\,\mbox{p.v.}\int_{-\pi}^\pi
\frac{f(e^{i\theta})e^{i\theta}}{e^{i\theta}-e^{i\vartheta}}d\theta
=
\widehat{f}(0)+i(\mathcal{C}f)(e^{i\vartheta}),
\]
where the operator $\mathcal{C}$, called the Hilbert transform,
is defined for $f\in L^1$ by
\begin{equation}\label{eq:Hilbert}
(\mathcal{C}f)\left(e^{i\vartheta}\right) 
:=
\frac{1}{2\pi}\,\mbox{p.v.}
\int_{-\pi}^{\pi}f\left(e^{i\theta}\right) 
\cot\frac{\vartheta-\theta}{2}\,d\theta , 
\quad \vartheta \in [-\pi, \pi].
\end{equation}
Hence the definition of $Pf$ for $f\in L^1$ in terms of the Cauchy singular
integral operator given by the second equality in \eqref{eq:Cauchy-Riesz} 
is equivalent to the following definition in terms of the Hilbert transform
and the zeroth Fourier coefficient of $f$ (cf. \cite[p.~104]{G06} and
\cite[Section~1.43]{BS06}):
\begin{equation}\label{eq:Hilbert-Riesz}
Pf:=\frac{1}{2}(f+i\mathcal{C}f)+\frac{1}{2}\widehat{f}(0).
\end{equation}
If $f\in L^1$ is such that $Pf\in L^1$, then
\begin{equation}\label{Rc1}
\widehat{Pf}(n)=\widehat{f}(n)
\quad\mbox{for}\quad 
n\ge 0,
\qquad
\widehat{Pf}(n)=0
\quad\mbox{for}\quad 
n< 0.
\end{equation}
Since we are not able to provide a precise reference to this
well known fact, we will give its proof in Subsection~\ref{Rcoeff}.
Note that definitions \eqref{eq:Cauchy-Riesz} can be extended to more general
Jordan curves in place of $\T$ (see, e.g., \cite{BK97} and also 
\cite{K02,K03,K17-MJOM}), while definitions \eqref{eq:Hilbert} and
\eqref{eq:Hilbert-Riesz} are used only in the case of the unit circle.
If $1<p<\infty$, then the operator $P$ projects $L^p$ onto $H^p$.
In view of this fact, the operator $P$ is usually called the Riesz projection.

For $a\in L^\infty$, the Toeplitz operator $T_a$ with symbol $a$ on $H^p$, 
$1<p<\infty$, is defined by
\[
T_af=P(af),
\quad
f\in H^p.
\]
The theory of Toeplitz operators has its origins in the classical paper by 
Otto Toeplitz \cite{T11}. Brown and Halmos \cite[Theorem~4]{BH63/64} proved 
that an operator on $H^2$ is a Toeplitz operator if and only if its matrix 
with respect to the standard basis is a Toeplitz matrix, that is, an infinite 
matrix of the form $(a_{j-k})_{j,k=0}^\infty$ (see also 
\cite[Part~B, Theorem~4.1.4]{N02} and \cite[Theorem~1.8]{PKh82}). 
An analogue of this result is true for Toeplitz operators acting on $H^p$, 
$1<p<\infty$ (see \cite[Theorem~2.7]{BS06}). 
Tolokonnikov \cite{T87} was the first to study Toeplitz operators acting 
between different Hardy spaces $H^p$ and $H^q$. In particular,  
\cite[Theorem~4]{T87} contains a description of all symbols generating 
bounded Toeplitz operators from $H^p$ to $H^q$ for $0<p,q\le\infty$.

Let $X$ be a Banach function space. We postpone the precise definition until 
Subsection~\ref{subsec:BFS}. For the moment, we observe only that it is 
continuously embedded in $L^1$. Following \cite[p.~877]{Xu92},  
we consider the abstract Hardy space $H[X]$ built upon the space 
$X$, which is defined by
\[
H[X]:=\big\{f\in X:\ \widehat{f}(n)=0\quad\mbox{for all}\quad n<0\big\} .
\]
It is clear that if $1\le p\le\infty$, then $H[L^p]$ 
is the classical Hardy space $H^p$.
\begin{lemma}\label{le:Riesz-projection}
If the operator $P$ defined by \eqref{eq:Cauchy-Riesz} is bounded on a 
Banach function space $X$ over the unit circle $\T$, then its image $P(X)$ 
coincides with the abstract Hardy space $H[X]$ built upon $X$.
\end{lemma}

Since $X\subset L^1$, this lemma follows immediately from formula \eqref{Rc1}
and the uniquiness theorem for Fourier series (see, e.g., 
\cite[Chap.~1, Theorem~2.7]{Kat76}).

Thus, the operator $P$ projects the Banach function space $X$ onto 
the abstract Hardy space $H[X]$. We will call $P$ the Riesz projection
as in the case of the spaces $L^p$ with $1<p<\infty$.

The Brown-Halmos theorem was extended by the first author to abstract Hardy 
spaces $H[X]$ built upon reflexive rearrangement-invariant Banach function 
spaces $X$ with non-trivial Boyd indices \cite[Theorem~4.5]{K04}. 
Under this assumption, the Riesz projection $P$ is bounded on $X$.
Further, it was shown in 
\cite[Theorem~1]{K17} that the Brown-Halmos theorem remains true for abstract 
Hardy spaces built upon arbitrarily, not necessarily rearrangement-invariant, 
reflexive Banach function spaces $X$ under the assumption that the Riesz 
projection is bounded on $X$. In particular, it is true for the weighted Hardy 
spaces $H^p(w)$, $1<p<\infty$, with Muckenhoput weights $w\in A_p(\T)$ 
\cite[Corollary~9]{K17}.

The space of all bounded linear operators from a Banach space $E$ to a Banach
space $F$ is denoted by $\cB(E,F)$. We adopt the standard abbreviation 
$\cB(E)$ for $\cB(E,E)$. We will write $E=F$ if $E$ and $F$ coincide as sets 
and there are constants $c_1,c_2\in(0,\infty)$ such that 
$c_1\|f\|_E\le \|f\|_F\le c_2\|f\|_E$ for all $f\in E$, and 
$E\equiv F$ if $E$ and $F$ coincide as sets and $\|f\|_E=\|f\|_F$
for all $f\in E$.

The aim of this paper is to study Toeplitz operators acting 
between abstract Hardy spaces $H[X]$ and $H[Y]$ built upon different Banach 
function spaces $X$ and $Y$ over the unit circle $\T$. We extend further
the results by Le\'snik \cite{L17}, who additionally assumed that the Banach 
function spaces $X$ and $Y$ are rearrangement-invariant. 
Let $L^0$ be the space of all measurable complex-valued functions on $\T$.
Following \cite{MP89}, let $M(X,Y)$ 
denote the space of pointwise multipliers from $X$ to $Y$ defined by 
\[
M(X,Y):=\{f\in L^0\ :\ fg\in Y\text{ for all } g\in X\}
\]
and equipped with the natural operator norm 
\[
\|f\|_{M(X,Y)}=\|M_f\|_{\cB(X,Y)}
=\sup_{\|g\|_X\le 1}\|fg\|_Y.
\]
Here $M_f$ stands for the operator of multiplication
by $f$ defined by $(M_fg)(t)=f(t)g(t)$ for $t\in\T$.

In particular, $M(X,X)\equiv L^\infty$. Note that it may happen that the 
space $M(X,Y)$ contains only the zero function.
For instance, if $1\le p<q\le\infty$, then $M(L^p,L^q)=\{0\}$.
The continuous embedding $L^\infty\subset M(X,Y)$ 
holds if and only if $X\subset Y$ continuously. For example, if 
$1\le q\le p\le\infty$, then $L^p\subset L^q$ and 
$M(L^p,L^q)\equiv L^{r}$, where $1/r=1/q-1/p$. 
For these and many other properties and examples, we refer to 
\cite{KLM13,LT17,MP89,N16} (see also references therein).

If the Riesz projection $P$ is bounded on the space $Y$, then one 
can define the Toeplitz operator $T_a$ with symbol $a\in M(X,Y)$ by
\[
T_af=P(af),
\quad
f\in H[X]
\]
(cf. \cite{L17}). It follows from Lemma~\ref{le:Riesz-projection}
that $T_af\in H[Y]$ and, clearly,
\[
\|T_a\|_{\cB(H[X],H[Y])}\le \|P\|_{\cB(Y)}\|a\|_{M(X,Y)}.
\]

Let $X'$ be the associate space of $X$ (see Subsection~\ref{subsec:BFS}).
For $f\in X$ and $g\in X'$, put
\[
\langle f,g\rangle:=
\int_\T f(t)\overline{g(t)}\,dm(t).
\]
For $n\in\Z$ and $\tau\in\T$, put $\chi_n(\tau):=\tau^n$. Then the Fourier 
coefficients of a function $f\in L^1$ can be expressed by
$\widehat{f}(n)=\langle f,\chi_n\rangle$ for $n\in\Z$. With these notation,
our main result reads as follows.
\begin{theorem}[\`a la Brown-Halmos]
\label{th:Brown-Halmos}
Let $X,Y$ be two Banach function spaces over the unit circle 
$\mathbb{T}$. Suppose that $X$ is separable and the
Riesz projection $P$ is bounded on the space 
$Y$.  If $A\in\cB(H[X],H[Y])$ and there exists a sequence $\{a_n\}_{n\in\Z}$ 
of complex numbers such that 
\begin{equation}\label{eq:Brown-Halmos-1}
\langle A\chi_j,\chi_k\rangle=a_{k-j} \quad\text{for all}\quad j,k\geq 0,
\end{equation}
then there is a function $a\in M(X,Y)$ such that $A=T_a$ and 
$\widehat{a}(n)=a_n$ for all $n\in\Z$. Moreover,
\begin{equation}\label{eq:Brown-Halmos-2}
\|a\|_{M(X,Y)}
\leq \|T_a\|_{\cB(H[X],H[Y])}\leq \|P\|_{\cB(Y)}\|a\|_{M(X,Y)}.
\end{equation}
\end{theorem}

Under the additional assumption that the Banach function spaces $X$ and $Y$ are
rearrangement-invariant, this result was recently obtained by Le\'snik
\cite[Theorem~4.2]{L17}.

The above theorem and the fact that $M(X,X)\equiv L^\infty$ 
(see \cite[Theorem~1]{MP89}) immediately imply the following.
\begin{corollary}
\label{co:Brown-Halmos}
Let $X$ be a separable Banach function spaces over the unit circle 
$\mathbb{T}$ and let the Riesz projection $P$ be bounded on $X$. 
If $A\in\cB(H[X])$ and there is a sequence $\{a_n\}_{n\in\Z}$ of 
complex numbers satisfying \eqref{eq:Brown-Halmos-1},
then there exists a function $a\in L^\infty$ such that $A=T_a$ and 
$\widehat{a}(n)=a_n$ for all $n\in\Z$. Moreover,
\[
\|a\|_{L^\infty}
\leq \|T_a\|_{\cB(H[X])}\leq \|P\|_{\cB(X)}\|a\|_{L^\infty}.
\]
\end{corollary}

Note that Corollary~\ref{co:Brown-Halmos} is also new. Under the additional
assumption that the Banach function space $X$ is reflexive, it was proved
by the first author in \cite[Theorem~1]{K17}. On the other hand, under the 
additional hypothesis that $X$ is rearrangement-invariant, it is established 
in \cite[Corollary~4.4]{L17}.

The paper is organized as follows. In Section~\ref{sec:preliminaries}, we 
collect preliminary facts on Banach function spaces $X$, including results on 
the density of the set of all trigonometric polynomials $\cP$ in $X$ and
the density of the set of all analytic polynomials $\cP_A$ in the abstract
Hardy space $H[X]$ built upon $X$. Further, we show that if 
each function in the closure 
$(X')_b$ of all simple functions in the associate space $X'$ 
has absolutely continuous norm, 
then the norm of any function $f\in X$ can be expressed as follows:
\begin{equation}\label{eq:NFP-introduction}
\|f\|_X=\sup\{|\langle f,p\rangle|\ :\ p\in\cP, \ \|p\|_{X'}\le 1\}.
\end{equation}
We conclude Section~\ref{sec:preliminaries} with several facts from complex 
analysis on the Hilbert transform and inner functions. In particular,
we recall a result by Qiu \cite[Lemma~5.1]{Q15} (see also
\cite[Theorem~7.2]{CGRT16}) saying that, for every measurable set 
$E\subset\T$ and an arc $\gamma\subset\T$ of the same measure, there exists 
an inner function $u$ such that $u^{-1}(\gamma)$ and $E$ coincide almost 
everywhere.

We start Section~\ref{sec:Riesz-consequences} on the consequences of the 
boundedness of the operator $P$ defined by \eqref{eq:Cauchy-Riesz}
with a discussion of operators of weak type. It is easy to see
that if the Riesz projection $P$ is bounded on $X$, then the Hilbert 
transform $\mathcal{C}$ is of weak types $(L^\infty,X)$ and $(L^\infty,X')$.
Using the existence of  the inner function $u$ mentioned 
above and properties of the Hilbert transform,
we show that if $\mathcal{C}$ is of weak types $(L^\infty,X)$ and 
$(L^\infty,X')$, then each function in
the closures $X_b$ and $(X')_b$ of the simple functions 
in $X$ and $X'$, respectively, has absolutely continuous norm. 
Thus, for every $f\in X$,  formula \eqref{eq:NFP-introduction} holds under 
the only assumption that $P\in\cB(X)$.

In Section~\ref{sec:proof}, we present a proof of Theorem~\ref{th:Brown-Halmos}.
Armed with the density of the set of analytic polynomials $\cP_A$ in the 
abstract Hardy space $H[X]$ built upon a separable Banach function space $X$ 
and formula \eqref{eq:NFP-introduction} with 
$Y$ such that $P\in\cB(Y)$ in place of $X$, we can adapt the proofs given in 
\cite[Theorem~2.7]{BS06} (for $X=Y=L^p$ with $1<p<\infty$) and in 
\cite[Theorem~4.2]{L17} (for the case of separable rearrangement-invariant 
spaces $X\subset Y$ such that $Y$ has non-trivial Boyd indices) to our setting.

In Section~\ref{sec:Nakano}, we specify the result of 
Theorem~\ref{th:Brown-Halmos} to the case of variable Lebesgue spaces
(also known as Nakano spaces) $X=L^{p(\cdot)}$ and $Y=L^{q(\cdot)}$.
It is known that if $1/q(t)=1/p(t)+1/r(t)$ for $t\in\T$, then 
$M(L^{p(\cdot)},L^{q(\cdot)})=L^{r(\cdot)}$ and that the Riesz projection 
$P$ is bounded on $L^{q(\cdot)}$ if the variable exponent $q$ is sufficiently 
smooth and bounded away from $1$ and $\infty$. Since the spaces $L^{p(\cdot)}$ 
and $L^{q(\cdot)}$ are not rearrangement-invariant, in general, the main result 
of Section~\ref{sec:Nakano} cannot be obtained from \cite[Theorem~4.2]{L17}.

In Section~\ref{sec:Lorentz}, we apply Corollary~\ref{co:Brown-Halmos}
to the case of Lorentz spaces $L^{p,q}(w)$, $1<p<\infty$, $1\le q<\infty$,
with Muckenhoupt weights $w\in A_p(\T)$. Under these assumptions,
$L^{p,q}(w)$ is a separable Banach function space and the Riesz projection
$P$ is bounded on $L^{p,q}(w)$. The space $L^{p,1}(w)$ is not reflexive and 
not rearrangement-invariant. Hence the earlier results of 
\cite[Theorem~1]{K17} and \cite[Corollary~4.4]{L17} are not applicable
to the space $L^{p,1}(w)$, while Corollary~\ref{co:Brown-Halmos} is.
\section{Preliminaries}\label{sec:preliminaries}
\subsection{Banach function spaces}\label{subsec:BFS}
Let $L^0_+$ be the
subset of functions in $L^0$ whose values lie in $[0,\infty]$. The
characteristic (indicator) function of a measurable set 
$E\subset\T$ is denoted by $\I_E$.

Following \cite[Chap.~1, Definition~1.1]{BS88}, a mapping 
$\rho: L_+^0\to [0,\infty]$ is called a Banach function norm
if, for all functions $f,g, f_n\in L_+^0$ with $n\in\N$, for all
constants $a\ge 0$, and for all measurable subsets $E$ of $\T$, the
following  properties hold:
\begin{eqnarray*}
{\rm (A1)} & &
\rho(f)=0  \Leftrightarrow  f=0\ \mbox{a.e.},
\
\rho(af)=a\rho(f),
\
\rho(f+g) \le \rho(f)+\rho(g),\\
{\rm (A2)} & &0\le g \le f \ \mbox{a.e.} \ \Rightarrow \ 
\rho(g) \le \rho(f)
\quad\mbox{(the lattice property)},\\
{\rm (A3)} & &0\le f_n \uparrow f \ \mbox{a.e.} \ \Rightarrow \
       \rho(f_n) \uparrow \rho(f)\quad\mbox{(the Fatou property)},\\
{\rm (A4)} & & m(E)<\infty\ \Rightarrow\ \rho(\I_E) <\infty,\\
{\rm (A5)} & &\int_E f(t)\,dm(t) \le C_E\rho(f)
\end{eqnarray*}
with {a constant} $C_E \in (0,\infty)$ that may depend on $E$ and 
$\rho$,  but is independent of $f$. When functions differing only on 
a set of measure  zero are identified, the set $X$ of all functions 
$f\in L^0$ for  which  $\rho(|f|)<\infty$ is called a Banach function
space. For each $f\in X$, the norm of $f$ is defined by
$\|f\|_X :=\rho(|f|)$.
The set $X$ under the natural linear space operations and under 
this norm becomes a Banach space (see 
\cite[Chap.~1, Theorems~1.4 and~1.6]{BS88}). 
If $\rho$ is a Banach function norm, its associate norm 
$\rho'$ is defined on $L_+^0$ by
\[
\rho'(g):=\sup\left\{
\int_\T f(t)g(t)\,dm(t) \ : \ 
f\in L_+^0, \ \rho(f) \le 1
\right\}, \ g\in L_+^0.
\]
It is a Banach function norm itself \cite[Chap.~1, Theorem~2.2]{BS88}.
The Banach function space $X'$ determined by the Banach function norm
$\rho'$ is called the associate space (K\"othe dual) of $X$. 
The associate space $X'$ can be viewed as a subspace of the (Banach) 
dual space $X^*$. 
\subsection{Density of polynomials}
For $n\in\Z_+:=\{0,1,2,\dots\}$, a function of the form  
$\sum_{k=-n}^n \alpha_k \chi_k$, where $\alpha_k\in\C$ 
for all $k\in\{-n,\dots,n\}$, is called a trigonometric polynomial of order 
$n$. The set of all trigonometric polynomials is denoted by $\cP$. 
Further, a function of the form $\sum_{k=0}^n \alpha_k \chi_k$ with
$\alpha_k\in\C$ for $k\in\{0,\dots,n\}$ is called an analytic
polynomial of order $n$. The set of all analytic polynomials is denoted by
$\cP_A$.

Following \cite[Chap.~1, Definition~3.1]{BS88}, a function $f$ in a Banach
function space $X$ is said to have absolutely continuous norm in $X$ if 
$\|f\I_{\gamma_n}\|_X\to 0$ for every sequence $\{\gamma_n\}_{n\in\N}$ of 
measurable sets such that $\I_{\gamma_n}\to 0$ almost everywhere as 
$n\to\infty$. The set of all functions of absolutely continuous norm in 
$X$ is denoted by $X_a$. If $X_a=X$, then one says that $X$ has absolutely 
continuous norm. Let $S_0$ be the set of all simple functions on $\T$. 
Following \cite[Chap.~1, Definition~3.9]{BS88}, let $X_b$ denote the closure 
of $S_0$ in the norm of $X$. 
\begin{lemma}\label{le:density-polynomials}
Let $X$ be a Banach function space over the unit circle $\T$. If $X_a=X_b$,
then the set of trigonometric polynomials $\mathcal{P}$ is dense in $X_b$.
\end{lemma}
\begin{proof}
The proof is analogous to the proof of \cite[Lemma~2.2.1]{K17-CM}. Assume that
$\mathcal{P}$ is not dense in $X_b$. Then, by a corollary of the Hahn-Banach
theorem (see, e.g., \cite[Chap.~7, Theorem~4.2]{BSU96}), there exists a 
nonzero functional $\Lambda\in (X_b)^*$ such that $\Lambda(p)=0$ for all
$p\in\mathcal{P}$. It follows from \cite[Chap.~1, Theorems~3.10 and~4.1]{BS88}
that if $X_a=X_b$, then $(X_b)^*=X'$. Hence there exists a nonzero function 
$h\in X'\subset L^1$ such that 
\[
\int_\T p(t)h(t)\,dm(t)=0\quad\mbox{for all}\quad p\in\mathcal{P}.
\]
Taking $p(t)=t^n$ for $n\in\Z$, we obtain that all Fourier coefficients of
$h\in L^1$ vanish, which implies that $h=0$ a.e. on $\T$ by the uniqueness
theorem of the Fourier series (see, e.g., \cite[Chap.~I, Theorem~2.7]{Kat76}).
This contradiction proves that $\cP$ is dense in $X_b$.
\qed
\end{proof}

Combining the above lemma with \cite[Chap.~1, Corollary~5.6 and 
Theorem~3.11]{BS88}, we arrive at the following well known result.
\begin{corollary}
\label{co:density-polynomials}
A Banach function space $X$ over the unit circle $\T$ is separable if and 
only if the set of trigonometric polynomials $\mathcal{P}$ is dense in $X$.
\end{corollary}

The analytic counterpart of the above result had a hard birth. First, observe 
that under the additional assumption that the Riesz projection $P$ is bounded 
on $X$, the density of the set of analytic polynomials $\cP_A$ in the
abstract Hardy space $H[X]$ trivially follows from 
\eqref{Rc1}, Lemma \ref{le:Riesz-projection}, and
Corollary~\ref{co:density-polynomials} (see \cite[Lemma~4]{K17}). 
Le\'snik \cite{L17a} conjectured that the boundedness of $P$ is 
superfluous here
and $\cP_A$ must be dense in the abstract Hardy space $H[X]$ under 
the hypothesis that $X$ is merely separable. 

If $X$ is a separable rearrangement-invariant Banach function space, then
\begin{equation}\label{eq:Fejer}
\|f*F_n-f\|_X\to 0
\quad\mbox{for every}\quad 
f\in X
\quad\mbox{as}\quad n\to\infty, 
\end{equation}
where $\{F_n\}$ is the sequence of the Fej\'er kernels on the unit circle $\T$.
The property in \eqref{eq:Fejer} implies the density of $\cP_A$ in $H[X]$
(see, e.g., \cite[Lemma~3.1(c)]{L17} or 
\cite[Theorem~1.0.1]{K17-CM}).
If $X$ is an arbitrary separable Banach function space, then \eqref{eq:Fejer} 
is true under the assumption that the Hardy-Littlewood maximal operator $M$ 
is bounded on its associate space $X'$ 
\cite[Theorem~3.2.1]{K17-CM}, whence $\cP_A$ 
is dense in $H[X]$ (see 
\cite[Theorem~1.0.2]{K17-CM}). Finally, in 
\cite[Theorem~1.4]{KS17-IWOTA} we constructed a separable weighted $L^1$ space 
$X$ such that \eqref{eq:Fejer} does not hold. On the other hand, we proved 
Le\'snik's conjecture. 
\begin{lemma}[{\cite[Theorem~1.5]{KS17-IWOTA}}]
\label{le:density-analytic-polynomials}
If $X$ is a separable Banach function space over the unit circle $\T$, then 
the set of analytic polynomials $\cP_A$ is dense in the abstract Hardy space 
$H[X]$ built upon the space $X$.
\end{lemma}
\subsection{Formulae for the norm in a Banach function space}
Let $X$ be a Banach function space over the unit circle $\T$ and $X'$ be its 
associate space. Then for every $f\in X$ and $h \in X'$, one has the following 
well known formulae:
\begin{align}
\|f\|_X &=\sup\{|\langle f,g\rangle|\ :\ g\in X',\ \|g\|_{X'}\le 1\},
\label{eq:NFP1-1}
\\
\|f\|_X &=\sup\{|\langle f,s\rangle|\ :\ s\in S_0,\ \|s\|_{X'}\le 1\}, 
\label{eq:NFP1-2}
\\
\|h\|_{X'} &=\sup\{|\langle h,s\rangle|\ :\ s\in S_0, \ \|s\|_X\le 1\}.
\label{eq:NFP-dual}
\end{align}
Equality~\eqref{eq:NFP1-1} follows from \cite[Chap.~1, Theorem~2.7 and 
Lem\-ma~2.8]{BS88}. Equality~\eqref{eq:NFP1-2} can be proved by a literal 
repetition of the proof of \cite[Lemma~2.10]{KS17-PLMS}. Equality 
\eqref{eq:NFP-dual} is obtained by applying formula \eqref{eq:NFP1-2} to 
$h\in X'$ and recalling that $X\equiv X''$ in view of the Lorentz-Luxembrug 
theorem (see \cite[Chap.~1, Theorem~2.7]{BS88}).
\begin{lemma}\label{le:NFP2}
Let $X$ be a Banach function space over the unit circle $\T$. If 
$(X')_a=(X')_b$, then for every $f\in X$,
\begin{equation}\label{eq:NFP2-1}
\|f\|_X=\sup\{|\langle f,p\rangle|\ :\ p\in\cP, \ \|p\|_{X'}\le 1\}.
\end{equation}
\end{lemma}
\begin{proof}
Since $\cP\subset X'$, equality~\eqref{eq:NFP1-1} immediately implies that
\begin{equation}\label{eq:NFP2-2}
\|f\|_X\ge \sup\{|\langle f,p\rangle|\ :\ p\in\cP, \ \|p\|_{X'}\le 1\}.
\end{equation}
Take any $g\in (X')_b$ such that $0<\|g\|_{X'}\le 1$. Since $(X')_a=(X')_b$,
it follows from Lemma~\ref{le:density-polynomials} that there is
a sequence $q_n\in\cP\setminus\{0\}$ such that 
$\|q_n-g\|_{X'}\to 0$ as $n\to\infty$.  
For $n\in\N$, put $p_n:=(\|g\|_{X'}/\|q_n\|_{X'})q_n\in\cP$.
Then, arguing as in \cite[Lemma~5]{K17}, one can show that
\begin{align*}
|\langle f,g\rangle|
=
\lim_{n\to\infty}|\langle f,p_n\rangle|
\le 
\sup_{n\in\N}|\langle f,p_n\rangle|
\le  
\sup\{|\langle f,p\rangle|\ :\ p\in\cP, \ \|p\|_{X'}\le 1\}.
\end{align*}
This inequality and equality \eqref{eq:NFP1-1} imply that
\begin{equation}\label{eq:NFP2-3}
\|f\|_X\le \sup\{|\langle f,p\rangle|\ :\ p\in\cP, \ \|p\|_{X'}\le 1\}.
\end{equation}
Combining inequalities \eqref{eq:NFP2-2} and \eqref{eq:NFP2-3},
we arrive at equality \eqref{eq:NFP2-1}.
\qed
\end{proof}

Note that Le\'snik proved formula \eqref{eq:NFP2-1} for arbitrary
rearrangement-invariant Banach function spaces $X$ (see \cite[Lemma~3.2]{L17}). 
His proof relies on the interpolation theorem of Calder\'on (see 
\cite[Chap.~3, Theo\-rem~2.2]{BS88}), which allows one to prove that for 
$f\in X'$, the sequence $p_n=f*F_n\in\mathcal{P}$ satisfies  
$\|p_n\|_{X'}\le\|f\|_{X'}$ for all
$n\in\mathbb{N}$. In the setting of 
arbitrary Banach function spaces, the tools based on interpolation are not 
available, but one can prove \eqref{eq:NFP2-1} for translation-invariant 
Banach  function spaces and their weighted generalizations with positive 
continuous weights (cf. \cite[Corollary 2.13]{KS17-PLMS}). In the next section, 
we  show that if the Riesz projection $P$ is bounded on a Banach function 
space $X$, then $(X')_a=(X')_b$,  whence formula 
\eqref{eq:NFP2-1} holds.
\subsection{Hardy spaces on the unit disk and inner functions}
Let $\mathbb{D}$ denote the open unit disk in the complex plane $\C$. Recall
that a function $F$ analytic in $\mathbb{D}$ is said to belong to the Hardy 
space $H^p(\mathbb{D})$, $0<p\le\infty$, if 
\[
\|F\|_{H^p(\mathbb{D})} :=
\sup_{0 \le r < 1}
\left(\frac{1}{2\pi}\int_{-\pi}^\pi |F(re^{i\theta})|^p\,d\theta\right)^{1/p} 
< \infty ,
\quad
0<p<\infty,
\quad
\|F\|_{H^\infty(\mathbb{D})} :=
\sup_{z \in \mathbb{D}} |F(z)| < \infty .
\]
Recall that an inner function is a
function $u\in H^\infty(\D)$ such that $|u(e^{i\theta})|=1$ for a.e.
$\theta\in[-\pi,\pi]$.

The following important fact was observed by Nordgren (see corollary
to \cite[Lemma~1]{N68} and also \cite[Remark~9.4.6]{CMR06}).
\begin{lemma}\label{le:Nordgren}
If $u$ is an inner function such that $u(0)=0$, then $u$ is a 
measure-preserving transformation from $\T$ onto itself.
\end{lemma}
\begin{proof}
We include a sketch of the proof for the readers' convenience. Let $G$ be an 
arbitrary measurable subset of $\mathbb{T}$ and let $h$ be the bounded harmonic 
function on $\mathbb{D}$ with the boundary values equal to $\I_G$. Then 
$h\circ u$ is the bounded harmonic function on $\mathbb{D}$ with the boundary 
values equal to $\I_{u^{-1}(G)}$, and
\[
m(G) 
= 
\frac{1}{2\pi}\int_{-\pi}^{\pi} \I_G\left(e^{i\theta}\right)\, d\theta 
= h(0) 
= h(u(0)) 
= 
\frac{1}{2\pi}\int_{-\pi}^{\pi} \I_{u^{-1}(G)}\left(e^{i\theta}\right)\, d\theta 
= 
m\left(u^{-1}(G)\right),
\]
which completes the proof.
\qed
\end{proof}

The next result is one of the most important ingredients in our proof.
It appeared in \cite[Lemma~5.1]{Q15} and \cite[Theorem~7.2]{CGRT16}.
\begin{theorem}\label{th:crucial-inner}
If $E\subset\T$ is a measurable set and $\gamma\subset\T$ is an arc such that
$m(E)=m(\gamma)$, then there exists an inner function $u$ satisfying
$u(0)=0$ and such that the sets $u^{-1}(\gamma)$ and $E$ are equal almost
everywhere. 
\end{theorem}
\subsection{The Hilbert transform and Poisson integrals}
For $\vartheta\in[-\pi,\pi]$ and $r\in[0,1)$, let
\[
P_r(\vartheta) := \frac{1 - r^2}{1 - 2r \cos \vartheta + r^2}, 
\quad  
Q_r(\vartheta) := \frac{2r \sin\vartheta}{1 - 2r \cos \vartheta + r^2}
\]
be the Poisson kernel and the conjugate Poisson kernel, respectively.
\begin{theorem}\label{th:Hilbert-transform-basics}
Let $1 < p < \infty$.
\begin{enumerate}
\item[{\rm(a)}]
If $f\in L^p$ is a real-valued function, then the function defined by 
\begin{equation}\label{eq:Hilbert-transform-basics-0}
u(re^{i\vartheta})=\frac{1}{2\pi}\int_{-\pi}^\pi
f(e^{i\theta})(P_r+iQ_r)(\vartheta-\theta)\,d\theta,
\quad
\vartheta\in[-\pi,\pi],\ r\in[0,1),
\end{equation}
belongs to the Hardy space $H^p(\mathbb{D})$. Its nontangential boundary values 
$u(e^{i\vartheta})$ as $z\to e^{i\vartheta}$ exist for a.e. 
$\vartheta\in[-\pi,\pi]$ and 
\begin{equation}\label{eq:Hilbert-transform-basics-1}
\operatorname{Re}u(e^{i\vartheta})=f(e^{i\vartheta}) ,  \quad 
\operatorname{Im}u(e^{i\vartheta})=(\mathcal{C}f)(e^{i\vartheta}) 
\quad\mbox{for a.e.}\quad\vartheta\in[-\pi,\pi],
\end{equation}
where $\mathcal{C}$ is the Hilbert transform defined by 
\eqref{eq:Hilbert}.

\item[{\rm(b)}]
If $u \in H^p(\mathbb{D})$ and $\operatorname{Im}u(0) = 0$, then there is a 
real-valued function $f\in L^p$  
such that \eqref{eq:Hilbert-transform-basics-0} holds.
\end{enumerate}
\end{theorem}

This statement is well known (see, e.g., \cite[Chap.~I, Section D and 
Chap.~V, Section B.$2^\circ$]{K98}). 
\subsection{Fourier coefficients of $Pf$: proof of formula \eqref{Rc1}}
\label{Rcoeff}
Since $f\in L^1$, the Cauchy integral
\[
F(z)=\frac{1}{2\pi i}\int_\mathbb{T}\frac{f(\tau)}{\tau-z}\,d\tau,
\quad z\in\mathbb{D},
\]
belongs to $H^p(\mathbb{D})$ for all $0<p<1$ (see, e.g., 
\cite[Theorem~3.5]{D70}). By Privalov's theorem (see, e.g., 
\cite[Chap.~X, \S 3, Theorem~1]{G69}), the nontangential limit of $F(z)$ 
as $z\to e^{i\vartheta}$ coincides with $(Pf)(e^{i\vartheta})$ for a.e. 
$\vartheta\in[-\pi,\pi]$. Hence, taking into account that $Pf\in L^1$, 
by Smirnov's theorem (see, e.g., \cite[Chap. IX, \S 4, Theorem~4]{G69} 
or \cite[Theorem~3.4]{D70}),  $F\in H^1(\mathbb{D})$. Then \eqref{Rc1} follows 
from \cite[Theorem~3.4]{D70} and the formula for the Taylor coefficients of $F$:
\[
\frac{1}{n!}\, F^{(n)}(0) 
= 
\frac{1}{2\pi i}\int_\mathbb{T}\frac{f(\tau)}{\tau^{n + 1}}\,d\tau 
=
\frac{1}{2\pi} \int_{-\pi}^\pi f(e^{i\varphi})e^{-in\varphi}\,d\varphi 
= 
\widehat{f}(n) ,   \ \ \   n \ge 0,
\]
which completes the proof.
\qed
\section{Consequences of the boundedness of the Riesz projection}
\label{sec:Riesz-consequences}
\subsection{Operators of weak type}
Let $X$ and $Y$ be Banach function spaces over the unit circle. Following 
\cite{B99}, we say that a linear operator $A: X \to L^0$ is of weak type 
$(X,Y)$ if there exists a constant $C > 0$ such that for all $\lambda > 0$ 
and  $f \in X$,
\begin{equation}\label{eq:weakXY}
\left\|
\I_{\left\{\zeta \in \mathbb{T} : \ |(Af)(\zeta)| > \lambda\right\}}
\right\|_{Y} 
\le 
C\, \frac{\|f\|_{X}}{\lambda}.
\end{equation} 
We denote the infimum of the constants $C$ satisfying \eqref{eq:weakXY} by 
$\|A\|_{\cW(Y, X)}$ and the set of all operators of weak type $(X,Y)$ by
$\cW(X,Y)$.
\begin{lemma}\label{le:strong-implies-weak}
Let $X,Y$ be Banach function spaces over the unit circle $\T$. If
$A\in\cB(X,Y)$, then $A\in\cW(X,Y)$ and $\|A\|_{\cW(X,Y)}\le\|A\|_{\cB(X,Y)}$.
\end{lemma}
\begin{proof}
For all $\lambda>0$, $f\in X$ and almost all $\tau\in\T$, one has
\[
\I_{\left\{\zeta \in \mathbb{T} : \ |(Af)(\zeta)| > \lambda\right\}}(\tau)
\le 
\I_{\left\{\zeta \in \mathbb{T} : \ |(Af)(\zeta)| > \lambda\right\}}(\tau)
\frac{|(Af)(\tau)|}{\lambda}
\le 
\frac{|(Af)(\tau)|}{\lambda}.
\]
It follows from the above inequality, the lattice property, and the 
boundedness of the operator $A$ that
\[
\left\|
\I_{\left\{\zeta \in \mathbb{T} : \ |(Af)(\zeta)| > \lambda\right\}}
\right\|_{Y}
\le  
\left\|\frac{Af}{\lambda}\right\|_Y 
\le 
\|A\|_{\cB(X,Y)}\frac{\|f\|_X}{\lambda},
\]
which completes the proof.
\qed
\end{proof}
\subsection{Pointwise estimate for the Hilbert transform}
For a set $G \subset [-\pi, \pi]$, we use the following notation
\[
\I^*_G\left(e^{i\theta}\right) := \left\{\begin{array}{rl}
  1 ,   &  \theta \in G , \\
   0 ,  &   \theta \in [-\pi, \pi] \setminus G.
\end{array}
\right.
\]
Let $|G|$ denote the Lebesgue measure of $G$.
\begin{lemma}\label{le:Shargorodsky} 
For every measurable set $E \subset  [-\pi, \pi]$ with $0<|E| \le \pi/2$, there 
exists a measurable set $F  \subset  [-\pi, \pi]$ with $|F| = \pi$
such that
\begin{equation}\label{eq:Shargorodsky-1}
\left|(\mathcal{C}\I^*_F)\left(e^{i\vartheta}\right)\right| 
>
\frac{1}{\pi}\, \left|\log\left(\sqrt{2}\, \sin\frac{|E|}{2}\right)\right| 
\quad \mbox{\rm for a.e.} \quad
\vartheta \in E .
\end{equation}
\end{lemma}
\begin{proof}
Let $\ell:=\{e^{i\eta}\in\T:\pi-|E|<\eta<\pi\}$. By 
Theorem~\ref{th:crucial-inner}, there exists an inner function $V$ such that 
$V(0)=0$ and
\begin{equation}\label{eq:Shargorodsky-2}
V\left(e^{i\vartheta}\right) \in 
\left\{ \begin{array}{lll}
    \ell & \mbox{for a.e.} & \vartheta \in E , 
    \\ 
  \mathbb{T}\setminus\ell  & \mbox{for a.e.} & 
  \vartheta \in [-\pi, \pi] \setminus E .
\end{array}\right.  
\end{equation}
Consider the set
\begin{equation}\label{eq:Shargorodsky-3}
F := \left\{\theta \in [-\pi, \pi] : \ 
\operatorname{Im}V\left(e^{i\theta}\right) \le 0\right\}.
\end{equation}
Since $V(0)=0$ and $V$ is inner, it defines a measure-preserving transformation
of $\T$ onto itself due to Lemma~\ref{le:Nordgren}. Therefore,
\[
|F|
= 
\left|\left\{\vartheta \in [-\pi, \pi] : \ 
\operatorname{Im} e^{i\vartheta} \le 0\right\}\right|
=\pi.
\]
For $\eta \in [-\pi, \pi]$ and  $r \in [0, 1)$, let
\[
w\left(r e^{i\eta}\right) := 
\frac{1}{2\pi} \int_{-\pi}^{\pi} \I^*_{[-\pi, 0]}\left(e^{i\zeta}\right) 
\left(P_r + i Q_r\right)(\eta-\zeta)\, d\zeta.
\]
By Theorem~\ref{th:Hilbert-transform-basics}, the function  $w \in H^2(\D)$ has  
nontangential boundary values $w(e^{i\eta})$ as
$z\to e^{i\eta}$ for a.e. $\eta\in[-\pi,\pi]$ and
\begin{align}
\operatorname{Re}w(e^{i\eta})
&=
\I_{[-\pi,0]}^*(e^{i\eta}) \quad\ \quad\mbox{for a.e.}\quad\eta\in[-\pi,\pi],
\label{eq:Shargorodsky-4}
\\
\operatorname{Im}w(e^{i\eta})
&=
\big(\mathcal{C}\I_{[-\pi,0]}^*\big)(e^{i\eta}) 
\quad\mbox{for a.e.}\quad\eta\in[-\pi,\pi].
\label{eq:Shargorodsky-5}
\end{align}
It is clear that for $\eta\in(\pi-|E|,\pi)$,
\begin{equation}\label{eq:Shargorodsky-6}
\big(\mathcal{C}\I_{[-\pi,0]}^*\big)(e^{i\eta})
=
\frac{1}{2\pi}\int_{-\pi}^0\cot\frac{\eta-\zeta}{2}\,d\zeta
=
\frac{1}{\pi}\log\sin\frac{\eta}{2}
-
\frac{1}{\pi}\log\sin\frac{\eta+\pi}{2}.
\end{equation}
Since $|E| \in(0,\pi/2]$, we have for all $\eta\in(\pi-|E|,\pi)$,
\begin{equation}
\log\sin\frac{\eta}{2}
>
\log\sin\frac{\pi}{4}
=
-\log\sqrt{2}
\ge 
\log\sin\frac{|E|}{2}
>
\log\sin\frac{\eta+\pi}{2}.
\label{eq:Shargorodsky-7}
\end{equation}
It follows from \eqref{eq:Shargorodsky-5}--\eqref{eq:Shargorodsky-7} that
for a.e. $\eta\in(\pi-|E|,\pi)$,
\begin{equation}\label{eq:Shargorodsky-8}
|\operatorname{Im}w(e^{i\eta})|
>
\frac{1}{\pi}
\left(-\log\sqrt{2}-\log\sin\frac{|E|}{2}\right) 
=
\frac{1}{\pi}\, \left|\log\left(\sqrt{2}\, \sin\frac{|E|}{2}\right)\right|.
\end{equation}
Consider now the function $W=w\circ V$, which belongs to $H^2(\D)$ (see, e.g., 
\cite[Section 2.6]{D70}). In view of \eqref{eq:Shargorodsky-3} and 
\eqref{eq:Shargorodsky-4}, we have
\[
\operatorname{Re}W(e^{i\vartheta})=\left\{\begin{array}{lll}
1 &\mbox{if}& \operatorname{Im}V(e^{i\vartheta})\le 0,
\\
0 &\mbox{if}& \operatorname{Im}V(e^{i\vartheta})>0
\end{array}\right.
=
\I_F^*(e^{i\vartheta})
\quad\mbox{for a.e.}\quad\vartheta\in[-\pi,\pi].
\]
Then, by Theorem~\ref{th:Hilbert-transform-basics}, 
\begin{equation}\label{eq:Shargorodsky-9}
\operatorname{Im}W(e^{i\vartheta})=(\mathcal{C}\I_F^*)(e^{i\vartheta})
\quad\mbox{for a.e.}\quad\vartheta\in[-\pi,\pi].
\end{equation}
If $\vartheta\in E$, then it follows from \eqref{eq:Shargorodsky-2}
that $V(e^{i\vartheta})\in\ell$. In this case inequality
\eqref{eq:Shargorodsky-8} implies that for a.e. $\vartheta\in E$,
\begin{equation}\label{eq:Shargorodsky-10}
|\operatorname{Im}W(e^{i\vartheta})|
=
\left|\operatorname{Im}w\left(V(e^{i\vartheta})\right)\right|
>
\frac{1}{\pi}\, \left|\log\left(\sqrt{2}\, \sin\frac{|E|}{2}\right)\right|.
\end{equation}
Combining equality \eqref{eq:Shargorodsky-9} and inequality
\eqref{eq:Shargorodsky-10}, we arrive at \eqref{eq:Shargorodsky-1}.
\qed
\end{proof}
\subsection{Equality $X_a=X_b$ if $\mathcal{C}\in\cW(L^\infty,X)$}
\begin{lemma}\label{le:norm-estimate} 
Let $X$  be a Banach function space over the unit circle $\T$. If the 
Hilbert transform $\mathcal{C}$ is of  weak type $(L^\infty, X)$, then for
every measurable set $E \subset  [-\pi, \pi]$ with $0<|E| \le \pi/2$,
one has
\begin{equation}\label{eq:norm-estimate}
\left\|\I^*_E\right\|_X
\le 
\frac{\pi \|\mathcal{C}\|_{\cW(L^\infty, Y)}}
{\left|\log\left(\sqrt{2}\, \sin\frac{|E|}{2}\right)\right|}.
\end{equation}
\end{lemma}
\begin{proof}
Let
\[
\lambda=
\frac{1}{\pi}
\left|\log\left(\sqrt{2}\, \sin\frac{|E|}{2}\right)\right|.
\]
By Lemma~\ref{le:Shargorodsky}, there exists a measurable set 
$F\subset [-\pi,\pi]$ with $|F|=\pi$ such that for a.e. $\tau\in\T$,
\[
\I_E^*(\tau)
\le 
\I_{\left\{\zeta\in\T\ :\ |(\mathcal{C}\I_F^*)(\zeta)|>\lambda\right\}}(\tau).
\]
Therefore, by the lattice property, taking into account that 
$\mathcal{C}\in\cW(L^\infty,X)$, we obtain
\[
\|\I_E^*\|_X
\le 
\left\|
\I_{\left\{\zeta\in\T\ :\ |(\mathcal{C}\I_F^*)(\zeta)|>\lambda\right\}}
\right\|_X
\le 
\frac{1}{\lambda}\|\mathcal{C}\|_{\cW(L^\infty,X)}
\|\I_F^*\|_{L^\infty}
=
\frac{\pi\|\mathcal{C}\|_{\cW(L^\infty,X)}}
{\left|\log\left(\sqrt{2}\, \sin\frac{|E|}{2}\right)\right|},
\]
which completes the proof.
\qed
\end{proof}
\begin{theorem}\label{th:Hilbert-weak-type-implies-Xa=Xb}
Let $X$  be a Banach function space over the unit circle $\T$. If the 
Hilbert transform $\mathcal{C}$ is of  weak type $(L^\infty, X)$, then
$X_a=X_b$.
\end{theorem}
\begin{proof}
Let $\Gamma\subset\T$ be a measurable set. Consider a sequence of measurable
subsets $\{\gamma_n\}_{n\in\N}$ of $\T$ such that $\I_{\gamma_n}\to 0$
a.e. on $\T$. By the dominated convergence theorem,
\[
m(\gamma_n)=\int_\T\I_{\gamma_n}(\tau)\,dm(\tau)\to 0
\quad\mbox{as}\quad n\to\infty.
\]
Without loss of generality, one can assume that $0<m(\gamma_n)\le 1/4$
for all $n\in\N$. For every $n\in\N$, there exists a measurable set 
$E_n\subset[-\pi,\pi]$ such that $\I_{\gamma_n}(\tau)=\I_{E_n}^*(\tau)$
for all $\tau\in\T$. It is clear that $|E_n|=2\pi m(\gamma_n)\le\pi/2$
for $n\in\N$. By Lemma~\ref{le:norm-estimate}, for every $n\in\N$,
\[
\|\I_\Gamma\I_{\gamma_n}\|_X
\le 
\|\I_{\gamma_n}\|_X=\|\I_{E_n}^*\|_X
\le 
\frac{\pi \|\mathcal{C}\|_{\cW(L^\infty, Y)}}
{\left|\log\left(\sqrt{2}\, \sin\frac{|E_n|}{2}\right)\right|}
=
\frac{\pi \|\mathcal{C}\|_{\cW(L^\infty, Y)}}
{\left|\log\left(\sqrt{2}\, \sin(\pi m(\gamma_n))\right)\right|}.
\]
Since $m(\gamma_n)\to 0$ as $n\to\infty$, the above estimate implies that 
$\|\I_\Gamma\I_{\gamma_n}\|_X\to 0$ as $n\to\infty$. Thus the function
$\I_\Gamma$ has absolutely continuous norm. By 
\cite[Chap.~1, Theorem~3.13]{BS88}, $X_a=X_b$.
\qed
\end{proof}
\subsection{Weak types $(L^\infty,X)$ and $(L^\infty,X')$ of the Hilbert 
transform if $P\in\cB(X)$}
\begin{lemma}\label{le:duality-argument}
Let $X$ be a Banach function space over the unit circle $\T$ and $X'$ be its
associate space. If $\mathcal{C}\in\cB(X_b,X)$, then 
$\mathcal{C}\in\cB((X')_b,X')$ and
\begin{equation}\label{eq:duality-argument-1}
\|\mathcal{C}\|_{\cB((X')_b,X')}
\le 
\|\mathcal{C}\|_{\cB(X_b,X)}.
\end{equation}
\end{lemma}
\begin{proof}
It is well known that the operator $i\mathcal{C}$ is a self-adjoint operator
on the space $L^2$ (see, e.g., \cite[Section~5.7.3(a)]{N02}). Therefore,
for all $s,v\in S_0\subset L^2$, one has
\begin{equation}\label{eq:duality-argument-2}
\langle\mathcal{C}v,s\rangle=-\langle v,\mathcal{C}s\rangle.
\end{equation}
It follows from 
equalities \eqref{eq:NFP-dual}, \eqref{eq:duality-argument-2},
and H\"older's inequality (see \cite[Chap.~1, Theorem~2.4]{BS88}) that for 
every $v\in S_0$,
\begin{align*}
\|\mathcal{C}v\|_{X'}
&=
\sup\{|\langle\mathcal{C}v,s\rangle|\ :\ s\in S_0,\ \|s\|_X\le 1\}
=
\sup\{|\langle v,\mathcal{C}s\rangle|\ :\ s\in S_0,\ \|s\|_X\le 1\}
\\
&\le 
\sup\{\|v\|_{X'}\|\mathcal{C}s\|_X\ :\ s\in S_0,\ \|s\|_X\le 1\}
\le 
\|\mathcal{C}\|_{\cB(X_b,X)}\|v\|_{X'}.
\end{align*}
Since $S_0$ is dense in $(X')_b$, we conclude that 
$\mathcal{C}\in\cB((X')_b,X')$ and \eqref{eq:duality-argument-1} holds.
\qed
\end{proof}
\begin{lemma}\label{le:Riesz-strong-implies-Hilbert-weak}
Let $X$ be a Banach function space over the unit circle $\T$ and $X'$ be its
associate space. If the Riesz projection $P$ is bounded on $X$, 
then $\mathcal{C}\in\cW(L^\infty,X)$ and $\mathcal{C}\in\cW(L^\infty,X')$.
\end{lemma}
\begin{proof}
Since $X$ is continuously embedded into $L^1$, the functional 
$f\mapsto \widehat{f}(0)$ is continuous on the space $X$. Then it follows 
from \eqref{eq:Hilbert-Riesz} that
$P\in\cB(X)$ if and only if
$\mathcal{C}\in\cB(X)$. Since $L^\infty$ is continuously embedded into $X$, 
one has  $\cB(X)\subset\cB(L^\infty,X)$. By Lemma~\ref{le:strong-implies-weak},
$\cB(L^\infty,X)\subset\cW(L^\infty,X)$. These observations imply that
$\mathcal{C}\in\cW(L^\infty,X)$ if $P\in\cB(X)$. Since $X_b$ is a 
Banach space isometrically embedded into $X$ (see 
\cite[Chap.~1, Theorem~3.1]{BS88}), we see that 
$\mathcal{C}\in\cB(X)\subset\cB(X_b,X)$ if $P\in\cB(X)$. Then, by
Lemma~\ref{le:duality-argument}, $\mathcal{C}\in\cB((X')_b,X')$.
Taking into account that $L^\infty$ is continuously embedded into
$(X')_b$ (see, e.g., \cite[Chap.~1, Proposition~3.10]{BS88}), 
we get $C\in\cB((X')_b,X')\subset\cB(L^\infty,X')$,
which implies that $\mathcal{C}\in\cW(L^\infty,X')$ in view of
Lemma~\ref{le:strong-implies-weak}.
\qed
\end{proof}
\subsection{Equalities $X_a=X_b$ and $(X')_a=(X')_b$ if $P\in\cB(X)$}
Now we are in a position to formulate the main result of this section.
\begin{theorem}\label{th:Riesz-implies-Xa=X_b-and-Xprime_a=Xprime_b}
Let $X$ be a Banach function space over the unit circle $\mathbb{T}$. 
If the Riesz projection $P$ is bounded on $X$, 
then $X_a=X_b$ and $(X')_a=(X')_b$.
\end{theorem}
\begin{proof}
If the Riesz projection $P$ is bounded on a Banach function space 
$X$, then the Hilbert transform $\mathcal{C}$ is of weak 
types $(L^\infty,X)$ and $(L^\infty,X')$
in view of Lemma~\ref{le:Riesz-strong-implies-Hilbert-weak}. In turn,
$\mathcal{C}\in\cW(L^\infty,X)$ implies that $X_a=X_b$
and $\mathcal{C}\in\cW(L^\infty,X')$ implies that $(X')_a=(X')_b$
due to Theorem~\ref{th:Hilbert-weak-type-implies-Xa=Xb}.
\qed
\end{proof}

Combining Theorem~\ref{th:Riesz-implies-Xa=X_b-and-Xprime_a=Xprime_b} and 
Lemma~\ref{le:NFP2}, we immediately arrive at the following.
\begin{corollary}\label{co:NFP3}
Let $X$ be a Banach function space over the unit circle $\T$. 
If the Riesz projection $P$ is bounded on $X$,
then for every $f\in X$,
\[
\|f\|_X=\sup\{|\langle f,p\rangle|\ :\ p\in\cP, \ \|p\|_{X'}\le 1\}.
\]
\end{corollary}
\section{Proof of the main result}
\label{sec:proof}
\subsection{Multiplication operators}
\begin{lemma}\label{le:multiplication-operator}
Let $X,Y$ be Banach functions spaces over the unit circle $\T$. Suppose $X$ 
is separable and $A\in\cB(X,Y)$. If there exists a sequence $\{a_n\}_{n\in\Z}$ 
of complex numbers such that
\begin{equation}\label{eq:multiplication-1}
\langle A\chi_j,\chi_k\rangle=a_{k-j}
\quad\text{for all}\quad
j,k\in\Z,
\end{equation}
then there exists a function $a\in M(X,Y)$ such that $A=M_a$ and
$\widehat{a}(n)=a_n$ for all $n\in\Z$.
\end{lemma}
\begin{proof}
This statement was proved in \cite[Lemma~4.1]{L17} under the 
additional hypothesis that $X$ and $Y$ are rearrangement-invariant Banach 
function spaces. 
Put $a:=A\chi_0\in Y$. Then, one can show exactly as in \cite{L17}
that $(af)\widehat{\hspace{2mm}}(j)=(Af)\widehat{\hspace{2mm}}(j)$ for all
$j\in\Z$ and $f\in\mathcal{P}$.
Therefore, $Af=af$ for all $f\in\mathcal{P}$ in view of the uniquiness theorem 
for Fourier series (see, e.g., \cite[Chap.~1, Theorem~2.7]{Kat76}).
 
Now let $f\in X$. Since the space $X$ is separable, the set
$\mathcal{P}$ is dense in $X$ by Corollary~\ref{co:density-polynomials}.
Then there exists a sequence $p_n\in\cP$ such that $p_n\to f$ in $X$ and, 
whence,  $Ap_n\to Af$  in $X$ as $n\to\infty$. By 
\cite[Chap.~1, Theorem~1.4]{BS88}, $p_n\to f$ and $Ap_n\to af$ in measure as 
$n\to\infty$. Then $ap_n\to af$ in measure as $n\to\infty$ (see, e.g., 
\cite[Corollary~2.2.6]{B07}). Hence, the sequence $Ap_n=ap_n$ converges in 
measure to the functions $Af$ and $af$ as $n\to\infty$. This implies that 
$Af$ and $af$ coincide a.e. on $\T$ (see, e.g., the discussion preceding 
\cite[Theorem~2.2.3]{B07}). Thus $Af=af$ for all $f\in X$. 
This means that $A=M_a$ and $a\in M(X,Y)$ by the definition of $M(X,Y)$.
\qed
\end{proof}
\subsection{Proof of Theorem~\ref{th:Brown-Halmos}}
The aim of this subsection is to present a proof of our extension of the 
Brown-Halmos theorem. Although it follows the scheme of the proof of 
\cite[Theorem~2.7]{BS06} with modifications that are necessary in the setting 
of different spaces $X$ and $Y$ (cf. \cite[Thorem~4.2]{L17}), it uses results 
obtained in this paper (e.g., 
Theorem~\ref{th:Riesz-implies-Xa=X_b-and-Xprime_a=Xprime_b}
and Corollary~\ref{co:NFP3}) and in \cite{KS17-IWOTA} (see 
Lemma~\ref{le:density-analytic-polynomials} above). We 
provide details for the sake of completeness.

Since $P\in\cB(Y)$, it follows from 
Theorem~\ref{th:Riesz-implies-Xa=X_b-and-Xprime_a=Xprime_b} that
$(Y')_a=(Y')_b$. Then, by Lemma~\ref{le:density-polynomials},
the set of trigonometric polynomials $\mathcal{P}$ is dense in $(Y')_b$.
Therefore, $(Y')_b$ is separable. It follows from \cite[Chap.~1, Theorems~3.11
and~4.1]{BS88} that $((Y')_b)^*=Y''$. On the other hand, by the 
Lorentz-Luxemburg theorem (see \cite[Chap.~1, Theorem~2.7]{BS88}), 
$Y''\equiv Y$. Thus, the Banach function space $Y$ is canonically 
isometrically isomorphic to the dual space $((Y')_b)^*$ of the separable 
Banach space $(Y')_b$.

For $n\ge 0$, put $b_n:=\chi_{-n}A\chi_n$.
Then $b_n\in Y$ and
\begin{equation}\label{eq:BH-proof-1}
\|b_n\|_Y
=
\|A\chi_n\|_Y
=
\|A\chi_n\|_{H[Y]}
\le 
\|A\|_{\cB(H[X],H[Y])} \|\chi_n\|_X
=
\|A\|_{\cB(H[X],H[Y])} \|1\|_X.
\end{equation}
Put
\[
V=\left\{y\in (Y')_b\ :\ \|y\|_{Y'}<
\frac{1}{\|A\|_{\cB(H[X],H[Y])} \|1\|_X}\right\}.
\]
It follows from the H\"older inequality (see \cite[Chap.~1, Theorem~2.4]{BS88})
and \eqref{eq:BH-proof-1} that
\[
|\langle b_n,y\rangle|\leq \|b_n\|_Y \|y\|_{Y'}<1
\quad\mbox{for all}\quad y\in V,\ n\ge 0.
\]
Applying a corollary of  the Banach-Alaoglu theorem (see, e.g., 
\cite[Theorem~3.17]{R91}) to the neighborhood $V$ of zero in the separable
Banach space $(Y')_b$ and the sequence 
$\{b_n\}_{n\in\N}\subset Y=((Y')_b)^*$, we deduce that there exists
a function $b\in Y$ such that some subsequence $\big\{b_{n_k}\big\}_{k\in\N}$
of $\{b_n\}_{n\in\N}$ converges to $b$ in the weak-* topology of $((Y')_b)^*$.
It follows from \cite[Chap.~1, Proposition~3.10]{BS88} that $\chi_j\in (Y')_b$
for all $j\in\Z$. Hence
\begin{equation}\label{eq:BH-proof-2}
\lim_{k\to+\infty}\langle b_{n_k},\chi_j\rangle =\langle b,\chi_j\rangle
\quad\mbox{for all}\quad j\in\Z.
\end{equation}
On the other hand, we get from the definition of $b_n$ and 
\eqref{eq:Brown-Halmos-1} for $n_k+j\ge 0$,
\begin{equation}\label{eq:BH-proof-3}
\langle b_{n_k},\chi_j\rangle
=
\langle\chi_{-n_k}A\chi_{n_k},\chi_j\rangle
=
\langle A\chi_{n_k},\chi_{n_k+j}\rangle
=
a_j.
\end{equation}
It follows from \eqref{eq:BH-proof-2} and \eqref{eq:BH-proof-3} that
\begin{equation}\label{eq:BH-proof-4}
\langle b,\chi_j\rangle =a_j
\quad\mbox{for all}\quad j\in\Z.
\end{equation}
Now define the mapping $B$ by
\begin{equation}\label{eq:BH-proof-5}
B:\cP\to Y,\quad f\mapsto bf.
\end{equation}
Assume that $f$ and $g$ are trigonometric polynomials of order $m$ and $r$, 
respectively. Using equalities \eqref{eq:Brown-Halmos-1} and \eqref{eq:BH-proof-4}
and definition \eqref{eq:BH-proof-5}, one can show that for 
$n\ge\max\{m,r\}$,
\begin{equation}\label{eq:BH-proof-6}
\langle Bf,g\rangle 
=
\langle\chi_{-n}A(\chi_n f),g\rangle.
\end{equation}
It is clear that for those $n$, one has $\chi_n f\in H[X]$. Since
$A\in\cB(H[X],H[Y])$, we obtain
\begin{equation}\label{eq:BH-proof-7}
\|A(\chi_n f)\|_Y
=
\|A(\chi_n f)\|_{H[Y]}
\leq
\|A\|_{\cB(H[X],H[Y])}\|\chi_n f\|_{H[X]}
=
\|A\|_{\cB(H[X],H[Y])}\|f\|_{X}.
\end{equation}
Hence, by the H\"older inequality (see
\cite[Chap.~1, Theorem~2.4]{BS88}), we deduce from \eqref{eq:BH-proof-7}
that
\begin{equation}\label{eq:BH-proof-8}
|\langle\chi_{-n}A(\chi_n f),g\rangle|
\le 
\|\chi_{-n}A(\chi_n f)\|_Y\|g\|_{Y'}
=
\|A(\chi_n f)\|_Y\|g\|_{Y'}
\le 
\|A\|_{\cB(H[X],H[Y])}\|f\|_X\|g\|_{Y'}.
\end{equation}
It follows from \eqref{eq:BH-proof-6} and \eqref{eq:BH-proof-8} that
\begin{align}
|\langle Bf,g\rangle|
&\le 
\limsup_{n\to\infty}|\langle\chi_{-n}A(\chi_n f),g\rangle|
\le 
\|A\|_{\cB(H[X],H[Y])}\|f\|_{X}\|g\|_{Y'}.
\label{eq:BH-proof-9}
\end{align}
Since the Riesz projection $P$ is bounded on $Y$, inequality 
\eqref{eq:BH-proof-9} and Corolary~\ref{co:NFP3} imply that for every 
$f\in\cP$,
\[
\|Bf\|_{Y}
=
\sup\{|\langle Bf,g\rangle|\ :\ g\in\cP,\ \|g\|_{Y'}\le 1\}
\le 
\|A\|_{\cB(H[X],H[Y])}\|f\|_X.
\]
Since $X$ is separable, the set $\cP$ is dense in $X$ in view of
Corollary~\ref{co:density-polynomials}. Hence the above 
inequality shows that the linear mapping defined in \eqref{eq:BH-proof-5} 
extends to an operator $B\in\cB(X,Y)$ with
\begin{equation}\label{eq:BH-proof-10}
\|B\|_{\cB(X,Y)}\le \|A\|_{\cB(H[X],H[Y])}.
\end{equation}
We deduce from \eqref{eq:BH-proof-4}--\eqref{eq:BH-proof-5} that
\[
\langle B\chi_j,\chi_k\rangle
=
\langle b\chi_j,\chi_k\rangle
=
\langle b,\chi_{k-j}\rangle=a_{k-j}
\quad\mbox{for all}\quad j,k\in\Z.
\]
Then, by Lemma~\ref{le:multiplication-operator}, there exists a function
$a\in M(X,Y)$ such that $B=M_a$ and $\widehat{a}(n)=a_n$ for all $n\in\Z$.
Moreover,
\begin{equation}\label{eq:BH-proof-11}
\|B\|_{\cB(X,Y)}=\|M_a\|_{\cB(X,Y)}=\|a\|_{M(X,Y)}.
\end{equation}
It follows from the definition of the Toeplitz operator $T_a$ that 
\[
\langle T_a\chi_j,\chi_k\rangle
=
\widehat{a}(k-j)
\quad\mbox{for all}\quad j,k\ge 0.
\]
Combining this identity with \eqref{eq:Brown-Halmos-1}, we obtain
\begin{equation}\label{eq:BH-proof-12}
\langle T_a\chi_j,\chi_k\rangle =a_{k-j}=\langle A\chi_j,\chi_k\rangle
\quad\mbox{for all}\quad j,k\ge 0.
\end{equation}
Since $T_a\chi_j,A\chi_j\in H[Y]\subset H^1$, it follows from 
\eqref{eq:BH-proof-12} and the uniquiness theorem for Fourier series (see, 
e.g., \cite[Chap.~1, Theorem~2.7]{Kat76}) that $T_a\chi_j=A\chi_j$
for all $j\ge 0$. Therefore,
\begin{equation}\label{eq:BH-proof-13}
T_a f=Af\quad\mbox{for all}\quad f\in\cP_A.
\end{equation}
By Lemma~\ref{le:density-analytic-polynomials}, $\cP_A$ is dense in
$H[X]$. This observation and \eqref{eq:BH-proof-13} imply that
$T_a=A$ on $H[X]$ and
\begin{equation}\label{eq:BH-proof-14}
\|T_a\|_{\cB(H[X],H[Y])}=\|A\|_{\cB(H[X],H[Y])}.
\end{equation}
Combining inequality \eqref{eq:BH-proof-10} with equalities
\eqref{eq:BH-proof-11} and \eqref{eq:BH-proof-14}, we arrive at the first
inequality in \eqref{eq:Brown-Halmos-2}. The second inequality in
\eqref{eq:Brown-Halmos-2} is obvious.
\qed
\begin{remark}
Let the functions $a\in M(X,Y)$ and $b\in Y$ be as in the above proof.
Since $\chi_0 \in L^\infty \subset X$ and $a \in M(X, Y)$, we have
$a = a \chi_0 \in Y \subset L^1$. On the other hand, 
$b \in Y \subset L^1$. Note that the functions $a,b \in L^1$ have equal 
Fourier coefficients (see \eqref{eq:BH-proof-4} and 
Lemma~\ref{le:multiplication-operator}) and hence coincide (see, e.g.,
\cite[Chap.~1, Theorem~2.7]{Kat76}). 
\end{remark}
\section{Applications to variable Lebesgue spaces}
\label{sec:Nakano}
\subsection{Variable Lebesgue spaces}
Let $\mathfrak{P}(\T)$ be the set of all 
measurable functions $p:\T\to[1,\infty]$. For $p\in\mathfrak{P}(\T)$, put
\[
T_\infty^{p(\cdot)} :=\{t\in\T : p(t)=\infty\}.
\]
For a function $f\in L^0$, consider
\[
\varrho_{p(\cdot)}(f)
:=
\int_{\T\setminus\T_\infty^{p(\cdot)}}|f(t)|^{p(t)}dm(t)
+\|f\|_{L^\infty(\T_\infty^{p(\cdot)})}.
\]
The variable Lebesgue space
$L^{p(\cdot)}$ is defined (see, e.g., \cite[Definition~2.9]{CF13})
as the set of all measurable functions
$f\in L^0$ such that $\varrho_{p(\cdot)}(f/\lambda)<\infty$
for some $\lambda>0$. This space is a Banach function space with respect
to the Luxemburg-Nakano norm given by
\[
\|f\|_{L^{p(\cdot)}}:=\inf\{\lambda>0: \varrho_{p(\cdot)}(f/\lambda)\le 1\}
\]
(see \cite[Theorems~2.17, 2.71 and Section~2.10.3]{CF13}). 
If $p\in\mathfrak{P}(\T)$ is constant, then $L^{p(\cdot)}$ is nothing
but the standard Lebesgue space $L^p$. Variable Lebesgue spaces are often 
called Nakano spaces. We refer to Maligranda's paper \cite{M11} for the role 
of Hidegoro Nakano in the study of variable Lebesgue spaces.

For $p\in\mathfrak{P}(\T)$, put
\[
p_-:=\operatornamewithlimits{ess\,inf}_{t\in\T}p(t),
\quad
p_+:=\operatornamewithlimits{ess\,sup}_{t\in\T}p(t).
\]
It is well known that the variable Lebesgue space $L^{p(\cdot)}(\T)$
is separable if and only if $p_+<\infty$ and is reflexive if and only if
$1<p_-,p_+<\infty$ (see, e.g., \cite[Theorem~2.78 and Corollary~2.79]{CF13}).

The following result was obtained by Nakai \cite[Example~4.1]{N16}
under the additional hypothesis 
\[
\sup_{t\in\T\setminus \T_\infty^{r(\cdot)}}r(t)<\infty
\]
(and in the more general setting of quasi-Banach variable 
Lebesgue spaces spaces over arbitrary measure spaces).
Nakai also mentioned in \cite[Remark~4.2]{N16} (without proof) that this 
hypothesis is superfluous. One can find its proof in the present form in
\cite[Theorem~4.8]{K17-MJOM}.
\begin{theorem}
\label{th:multiplier-space-VLS}
Let $p,q,r\in\mathfrak{P}(\T)$ be related by 
\begin{equation}\label{eq:Hoelder}
\frac{1}{q(t)}=\frac{1}{p(t)}+\frac{1}{r(t)},
\quad t\in\T.
\end{equation}
Then 
$M(L^{p(\cdot)},L^{q(\cdot)})=L^{r(\cdot)}$.
\end{theorem}
\subsection{The Riesz projection on variable Lebesgue spaces}
We say that an exponent $q\in\mathfrak{P}(\T)$ is log-H\"older
continuous (cf. \cite[Definition~2.2]{CF13}) if $1<q_-\le q_+<\infty$
and there exists a constant $C_{q(\cdot)}\in(0,\infty)$ such that
\[
|q(t)-q(\tau)|\le\frac{C_{q(\cdot)}}{-\log|t-\tau|}
\quad\mbox{for all}\quad t,\tau\in\T\quad\mbox{satisfying}\quad |t-\tau|<1/2.
\]
The class of all log-H\"older continuous exponent will be denoted 
by $LH(\T)$. Some authors denote this class by 
$\mathbb{P}^{\log}(\T)$ (see, e.g., \cite[Section~1.1.4]{KMRS16}).
The following result is well known (see, e.g., 
\cite[Section~10.1]{KMRS16} or \cite[Lemma~12]{K17}).
\begin{theorem}\label{th:P-boundedness-VLS}
If $q\in LH(\T)$, then the Riesz projection $P$ is bounded on $L^{q(\cdot)}$. 
\end{theorem}
\subsection{Toeplitz operators between abstract Hardy space built upon
variable Lebesgue spaces}
Applying Theorems~\ref{th:Brown-Halmos}, \ref{th:multiplier-space-VLS}, and 
\ref{th:P-boundedness-VLS}, we arrive at the following.
\begin{theorem}
Let $p,q,r\in\mathfrak{P}(\T)$ be related by \eqref{eq:Hoelder}. Suppose
$q\in LH(\T)$ and $p_+<\infty$. If a linear operator $A$ is bounded form 
$H[L^{p(\cdot)}]$ to $H[L^{q(\cdot)}]$ and there exists a sequence 
$\{a_n\}_{n\in\Z}$ of complex numbers such that 
\[
\langle A\chi_j,\chi_k\rangle=a_{k-j} \quad\text{for all}\quad j,k\geq 0,
\]
then there is a function $a\in L^{r(\cdot)}$ such that $A=T_a$ and 
$\widehat{a}(n)=a_n$ for all $n\in\Z$. Moreover, there exist constants
$c_{p,q},C_{p,q}\in(0,\infty)$ depending only on $p$ and $q$ such that
\[
c_{p,q}\|a\|_{L^{r(\cdot)}}
\leq 
\|T_a\|_{\cB(H[L^{p(\cdot)}],H[L^{q(\cdot)}])}
\leq 
C_{p,q}
\|P\|_{\cB(L^{q(\cdot)})}\|a\|_{L^{r(\cdot)}}.
\]
\end{theorem}

Note that if $p,q\in LH(\T)$ coincide, then the constants $c_{p,q}$ and 
$C_{p,q}$ in the above inequality are equal to one 
(cf. \cite[Corollary~13]{K17}).
\section{Applications to Lorentz spaces with Muckenhoupt weights}
\label{sec:Lorentz}
\subsection{Rearrangement-invariant Banach function spaces}
The distribution function $m_f$ of an a.e. finite function $f\in L^0$ is 
given by
\[
m_f(\lambda) :=  m\{t\in\T:|f(t)|>\lambda\},\quad\lambda\ge 0.
\]
The non-increasing rearrangement of an a.e. finite function $f\in L^0$ is 
defined by 
\[
f^*(x):=\inf\{\lambda:m_f(\lambda)\le x\},\quad x\in[0,1].
\]
We refer to \cite[Chap.~2, Section~1]{BS88} for properties of distribution 
functions and non-increasing rearrangements.

Two a.e. finite functions $f,g\in L^0$ are said to be equimeasurable
if their distribution functions coincide: $m_f(\lambda)=m_g(\lambda)$ for all 
$\lambda\ge 0$.  A Banach function space $X$ over the unit circle $\T$ is 
called re\-ar\-range\-ment-invariant if for every pair of equimeasurable 
functions $f,g \in L^0$, $f\in X$ implies that $g\in X$ and the equality  
$\|f\|_X=\|g\|_X$ holds. For a rearrangement-invariant Banach function space 
$X$, its associate space $X'$ is also rearrangement-invariant
(see \cite[Chap.~2, Proposition~4.2]{BS88}.
\subsection{Lorentz spaces $L^{p,q}$}
Let $f$ be an a.e. finite function in $L^0$. For $x\in(0,1]$, put
\[
f^{**}(x)=\frac{1}{x}\int_0^x f^*(y)\,dy.
\]
Suppose $1<p<\infty$ and $1\le q\le\infty$. The Lorentz space $L^{p,q}$ 
consists of all a.e. finite functions $f\in L^0$ for which the quantity
\[
\|f\|_{L^{p,q}}=\left\{\begin{array}{lll}
\displaystyle
\left(\int_0^1 \left( x^{1/p}f^{**}(x)\right)^q\frac{dx}{x}\right)^{1/q},
&\mbox{if}& 1\le q<\infty,
\\
\displaystyle
\sup_{0<x<1}\left(x^{1/p}f^{**}(x)\right),
&\mbox{if}& q=\infty,
\end{array}\right.
\]
is finite. It is well known that $L^{p,q}$ is a rearrangement-invariant Banach 
function space with respect to the norm $\|\cdot\|_{L^{p,q}}$  (see, e.g., 
\cite[Chap.~4, Theorem~4.6]{BS88}, where the case of spaces of infinite 
measure is considered; in the case of spaces of finite measure, the proof 
is the same). It follows from \cite[Chap.~2, Proposition~1.8 and Chap.~4, 
Lemma~4.5]{BS88} that $L^{p,p}=L^p$ (with equivalent norms).
\subsection{Weighted Lorentz spaces $L^{p,q}(w)$}
For $q\in[1,\infty]$, put $q'=q/(q-1)$ with the usual conventions $1/0=\infty$ 
and $1/\infty=0$. A function $w\in L^0_+$ is referred to as a weight if 
$0<w(\tau)<\infty$ for a.e. $\tau\in\T$.

Let $1<p<\infty$ and $1\le q\le\infty$.
Suppose $w:\T\to[0,\infty]$ is a weight such that $w\in L^{p,q}$ and 
$1/w\in L^{p',q'}$. The weighted Lorentz space $L^{p,q}(w)$ is
defined as the set of all a.e. finite functions $f\in L^0$ such that 
$fw\in L^{p,q}$.

The next lemma follows directly from well known results on Lorentz spaces.
\begin{lemma}\label{le:separability-reflexivity-weighted-Lpq}
Let $1<p<\infty$, $1\le q\le\infty$ and $w:\T\to[0,\infty]$ be a weight 
such that $w\in L^{p,q}$, $1/w\in L^{p',q'}$.
\begin{enumerate}
\item[{\rm (a)}]
The space $L^{p,q}(w)$ is a Banach function space with respect 
to the norm $\|f\|_{L^{p,q}(w)}=\|fw\|_{L^{p,q}}$
and $L^{p',q'}(1/w)$ is its associate space.

\item[{\rm(b)}]
If $1<q<\infty$, then the space $L^{p,q}(w)$ is reflexive.

\item[{\rm(c)}]
The space $L^{p,1}(w)$ is separable and non-reflexive.
\end{enumerate}
\end{lemma}
\begin{proof}
(a) In view of \cite[Chap.~4, Theorem~4.7]{BS88}, the associate space of
the Lorentz space $L^{p,q}$, up to equivalence of norms, is the Lorentz space 
$L^{p',q'}$.  
It is easy to check that $L^{p,q}(w)$ is a Banach function space
and $L^{p',q'}(1/w)$ is its associate space. 

(b) Note that $L^{p,q}(w) \ni f \mapsto wf \in L^{p,q}$
is an isometric isomorphism of $L^{p,q}(w)$ and $L^{p,q}$. Hence these spaces 
have the same Banach space theory properties, e.g., reflexivity and 
separability.
If $1<p,q<\infty$, then $L^{p,q}$ is reflexive in view of 
\cite[Chap.~4, Corollary~4.8]{BS88}. 
Then the weighted Lorentz space $L^{p,q}(w)$ is reflexive too.

(c) If $1<p<\infty$, then $L^{p,1}$ has absolutely continuous norm
and $(L^{p,1})^*=L^{p',\infty}$ (see \cite[Chap.~4, Corollary~4.8]{BS88}).
Then $L^{p,1}$ is separable in view of \cite[Chap.~1, Corollary~5.6]{BS88}.
It is known that 
\[
L^{p,1}\subsetneqq (L^{p',\infty})^*=(L^{p,1})^{**}
\]
(see \cite[p.~83]{C75}). Hence $L^{p,1}$ is non-reflexive. 
Therefore, $L^{p,1}(w)$ is also separable and nonrefexive.
\qed
\end{proof}
\subsection{The Riesz projection on $L^{p,q}(w)$
with $1<p<\infty$, $1\le q<\infty$ and $w\in A_p(\T)$}
Let $1<p<\infty$ and $w$ be a weight. It is well known that the Riesz 
projection $P$ is bounded on the weighted Lebesgue space 
$L^p(w):=\{f\in L^0: fw\in L^p\}$ if and only 
if the weight $w$ satisfies the Muckenhoupt $A_p$-condition, that is,
\[
\sup_{\gamma\subset\T}
\left(\frac{1}{m(\gamma)}\int_\gamma 
w^p(\tau)\,dm(\tau)\right)^{1/p}
\left(\frac{1}{m(\gamma)}\int_\gamma 
w^{-p'}(\tau)\,dm(\tau)\right)^{1/p'}<\infty,
\]
where the supremum is taken over all subarcs $\gamma$ of the unit circle $\T$
(see \cite{HMW73} and also \cite[Section~6.2]{BK97}, \cite[Section~1.46]{BS06},  
\cite[Section~5.7.3(h)]{N02}). In this case, we will write $w\in A_p(\T)$.
\begin{lemma}\label{le:Ap-implies-weights-in-Lorentz}
Let $1<p<\infty$ and $1\le q\le\infty$. If $w\in A_p(\T)$, then 
$w\in L^{p,q}$ and $1/w\in L^{p',q'}$.
\end{lemma}
\begin{proof}
By the stability property of Muckenhoupt weights (see, e.g., 
\cite[Theorem~2.31]{BK97}), there exists $\eps>0$ such that 
$w\in A_s(\T)$ for all $s\in (p-\eps,p+\eps)$. Therefore, $w\in L^s$
and $1/w\in L^{s'}$ for all $s\in(p-\eps,p+\eps)$. In particular, if 
$s_1,s_2$ are such that $p-\eps<s_1<p<s_2<p+\eps$, then
$w\in L^{s_2}=L^{s_2,s_2}\subset L^{p,q}$ and
$1/w\in L^{s_1'}=L^{s_1',s_1'}\subset L^{p',q'}$ in view of the embeddings
of Lorentz spaces (see, e.g., \cite[Chap.~4, remark after Proposition~4.2]{BS88}).
\qed
\end{proof}

Lemmas~\ref{le:separability-reflexivity-weighted-Lpq}(a) 
and~\ref{le:Ap-implies-weights-in-Lorentz} imply that if $w\in A_p(\T)$, then
$L^{p,q}(w)$ is a Banach function space.
\begin{theorem}\label{th:P-boundedness-weighted-Lorentz}
Let $1<p<\infty$ and $1\le q\le\infty$.
If $w\in A_p(\T)$, then the Riesz projection $P$ is bounded on the weighted
Lorentz space $L^{p,q}(w)$.
\end{theorem}
\begin{proof}
It follows from \cite[Chap.~4, Theorem~4.6]{BS88} and \cite[Theorem~4.5]{K02}
that the Cauchy singular integral operator $S$ is bounded on $L^{p,q}(w)$.
Thus, the Riesz projection $P$ is bounded on $L^{p,q}(w)$ in view of 
\eqref{eq:Cauchy-Riesz}.
\qed
\end{proof}
\subsection{Toeplitz operators on abstract Hardy spaces built upon $L^{p,q}(w)$
with $1<p<\infty$, $1\le q<\infty$, $w\in A_p(\T)$}
The next theorem is an immediate consequence of Corollary~\ref{co:Brown-Halmos},
Lemmas~\ref{le:separability-reflexivity-weighted-Lpq} and 
\ref{le:Ap-implies-weights-in-Lorentz},
and Theorem~\ref{th:P-boundedness-weighted-Lorentz}.
\begin{theorem}
Let $1<p<\infty$, $1\le q<\infty$, and $w\in A_p(\T)$. If an operator $A$ is 
bounded on the abstract Hardy space $H[L^{p,q}(w)]$ and there exists a 
sequence $\{a_n\}_{n\in\Z}$ of complex numbers such that 
\[
\langle A\chi_j,\chi_k\rangle=a_{k-j} \quad\text{for all}\quad j,k\geq 0,
\]
then there is a function $a\in L^\infty$ such that $A=T_a$ and 
$\widehat{a}(n)=a_n$ for all $n\in\Z$. Moreover, 
\[
\|a\|_{L^\infty}
\leq 
\|T_a\|_{\cB(H[L^{p,q}(w)])}
\leq 
\|P\|_{\cB(L^{p,q}(w))}\|a\|_{L^\infty}.
\]
\end{theorem}

For $p=q$ this result is contained in \cite[Corollary~9]{K17}. 
For $1<q<\infty$, this result as well follows from \cite[Theorem~1]{K17}.
The most interesting case is when $q=1$ because in this case the weighted
Lorentz space $L^{p,1}(w)$ is separable and non-reflexive. Moreover, it is not
rearrangement-invariant. Therefore \cite[Theorem~1]{K17} and
\cite[Corollary~4.4]{L17} are not applicable, while
Corollary~\ref{co:Brown-Halmos} works in this case.
\section*{Acknowledgement}
We would like to thank the anonymous referee for several useful remarks 
and pointing out a mistake in the first version of the paper.
This work was partially supported by the Funda\c{c}\~ao para a Ci\^encia e a
Tecnologia (Portu\-guese Foundation for Science and Technology)
through the project
UID/MAT/00297/2013 (Centro de Matem\'atica e Aplica\c{c}\~oes).

\section*{References}

\end{document}